\documentclass{jaums}
\usepackage{amsmath}
\usepackage{amsfonts}
\usepackage{latexsym}
\usepackage{amssymb}
\usepackage[arrow, matrix, curve]{xy}
\usepackage{enumerate}
\usepackage[dvips]{graphics}

\theoremstyle{thmit} 
\newtheorem{thm}{Theorem}[section]

\newtheorem{cor}[thm]{Corollary}
\newtheorem{prop}[thm]{Proposition}
\theoremstyle{thmrm} 
\newtheorem{exa}{Example}

\newtheorem*{oldproof}{Proof}
\renewenvironment{proof}[1][{}]{\begin{oldproof}[#1]}{\qed\end{oldproof}}
\newtheorem{definition}{Definition}

\newcommand{\C}{{\mathbf C}}
\newcommand{\R}{{\mathbf R}}
\newcommand{\Z}{{\mathbf Z}}

\newcommand{\rk}{{\rm {rk}}}
\newcommand{\coker}{{\rm coker}}
\newcommand{\ind}{{\rm ind}}
\newcommand{\MM}{{\frak {m}}}
\title[Telescopic linkages and phase transitions]
 {Telescopic linkages and topological approach to phase transitions}

\author{Michael Farber and Viktor Fromm}        
\date{October 5, 2010}          

\address{Department of Mathematical Sciences, University of Durham, UK}

\email{Michael.Farber@durham.ac.uk}
\email{viktor.fromm@durham.ac.uk}

\begin{document}
\maketitle
\begin{abstract} A topological approach to the theory of equilibrium phase transitions in statistical physics is based on the {\it Topological Hypothesis} (TH), which claims that 
phase transitions are due to changes of the topology of suitable submanifolds in the configuration space (\cite{CPC2003, Pettini}). In this paper we examine in detail the anti-ferromagnetic mean-field $XY$ model and study topology of the sub-energy manifolds. The latter can be interpreted mechanically as the configuration space of a linkage with one telescopic leg. We use methods of Morse theory to describe explicitly the Betti numbers of this configuration space. We apply these results to the 
anti-ferromagnetic mean-field $XY$ model and compute the exponential growth rate of the total Betti number. The previous authors instead of the total Betti number studied the Euler characteristic. We show that in the presence of an external magnetic field the model undergoes a single \lq\lq total Betti number phase transition\rq\rq.
\end{abstract}

\begin{center}
\it Dedicated to Alan Carey, on the occasion of his 60$\,^{th}$ birthday
\end{center}

\section{ Thermodynamical phase transitions
and topological changes in configuration spaces }\label{s:1}

Equilibrium phase transitions are non-analytic points of thermodynamic observables. Modern physics has accumulated experimental evidence that some new mathematical mechanisms might be relevant to phase transitions in certain systems. 
A recent mathematical approach to phase transitions is based on the {\it Topological Hypothesis} (TH) which claims that at their deepest level phase transitions are due to changes of the topology of suitable submanifolds in the configuration space, see for example \cite{CPC2003, Pettini} and references therein.

 One knows that for a system with the Hamiltonian
 $$H= \frac{1}{2} \sum_{i=1}^N p_i^2 +V(q_1, \dots, q_N),$$
assuming that $N$ is large,
 at any given value of the inverse temperature $\beta$, the effective support of the canonical measure is very close to a single
 equipotential hypersurface
 $$\Sigma_v\equiv\{q\in \Gamma_N: V(q)=vN\}.$$ Here
  $\Gamma_N$ is the {\it configuration space} and
 $V: \Gamma_N \to \R$ is the {\it potential}.
 The TH claims that changes of the topology of the hypersurfaces $\Sigma_v$ or of the manifolds with boundary
 \begin{eqnarray}\label{mv}
\mathcal M_v=\{q\in \Gamma_N; V(q) \le vN\}
\end{eqnarray}
 (for $N\to \infty$) are the reason for the singular behavior of thermodynamic observables in phase transitions.
 One such observable is the configurational canonical free energy
 $$f_N(\beta) = -\frac{1}{N\beta} \ln \int_{\Gamma_N} e^{-\beta V(q)}dq.$$
 In \cite{Angelani2005}, \cite{Kastner08}
 the following quantity
 \begin{eqnarray}\label{sigmav}
 \sigma(v) = \lim_{N\to \infty} \frac{1}{N} \ln | \chi(M_v)|\end{eqnarray}
 involving the Euler characteristic $\chi(M_v)$ is studied as a function of $v$ and its non-smoothness is related to the phase transitions.
 Recall that in the case when the potential $V$ is Morse, the Euler characteristic $\chi(M_v)$ equals the sum $$\sum_{i\ge 0} (-1)^i \mu_i (V)$$ where $\mu_i(V)$ denotes the number of critical points of $V$ of Morse index $i$ lying in $M_v$. In the case when $V$ is Morse-Bott there is a similar formula
 $$\chi(M_v) = \sum_Z (-1)^{\ind(Z)} \chi(Z),$$ where $Z$ runs over the critical submanifolds of $V$, see \cite[Corollary 5.3]{Farber2004}.

 In  \cite{FP2004} a version of
 the TH was proven for a class of short-range models (Franzosi-Pettini Theorem) with potentials of the form
 $$V(q) = \sum_{i=1}^N \phi(q_i) + \sum_{i,j=1}^N  c_{ij}\psi(||q_i-q_j||),$$
 where $|\psi (x)|$ decreases faster than $x^{-d}$ with $d$ denoting the spatial dimension of the system.

For the mean-field $k$-trigonometric model (\cite{Angelani2005}), which is characterized by the potential
 $$V_k(q) = \frac{1}{N^{k-1}} \sum_{i_1, \dots, i_k=1}^N [1-\cos(q_{i_1}+\dots +q_{i_k})],$$
where $q_i\in [0,2\pi]$ are the angular variables, 
 it was shown that phase transitions occur for $k\ge 2$ and the function $\sigma(v)$ in (\ref{sigmav}), is smooth for $k=1$ and nonsmooth for $k\ge 2$. 
Moreover,  for $k\ge 2$, $\sigma(v)$ is nonsmooth precisely at
 the values of the energy at which the phase transitions occur.  Thus, the purely topological quantity $\sigma(v)$ signals the absence or presence of a phase transition. This observation as well as similar findings for other models motivate the further study of the relations between thermodynamic phase transitions and the topology of the family
 $\mathcal M_v$. In the literature there are also results of negative character   concerning the TH, see \cite{Kastner04, Teixiera}.

\section{The anti-ferromagnetic mean-field $XY$-model and the robot arm}\label{s:2}

Christian Mazza pointed out on a similarity between the study of the robot arm in \cite[\S\S 1.5 and 3.7]{Fa3} and
the anti-ferromagnetic mean-field $XY$ model (\cite{LCJ, DHR}). 
This observation was our point of departure; it motivated us to study a new class of linkages
having one telescopic leg (i.e. a leg with length variable in certain interval) \cite{Fa4}; more details are given below.

The anti-ferromagnetic mean-field $XY$-model has
 potential of the form
\begin{eqnarray}\label{potential}
V = \frac{1}{2N} \sum_{i,j} \cos(\theta_i-\theta_j) -h \sum_i \sin \theta_i,\end{eqnarray}
 where $\theta_i\in [0, 2\pi]$, $i=1, \dots, N$, are angular parameters (classical rotators).
 Here $h$ denotes {\it an external magnetic field.} 

If $\MM$ denotes the complex magnetization vector
 \begin{eqnarray}\label{magnet}
 \MM=\frac{1}{N} \sum_{j=1}^N \exp{i\theta_j}
 \end{eqnarray} 
and $$\MM_0=-ih\in \C,$$ then 
%
%
\begin{eqnarray}\label{vm}
\left|\MM+\MM_0\right|^2 = \frac{2}{N}V +h^2.
\end{eqnarray}
We see that the sub-level set $\mathcal M_v$ given by (\ref{mv}) 
coincides with the set 
\begin{eqnarray}
{\mathcal M}_v \, =\, \{q; |\MM+\MM_0|^2 \le 2v +h^2\}.\end{eqnarray}
In this case $\Gamma_N$ is the $N$-dimensional torus $S^1 \times \dots\times S^1=T^N$; the symbol $q$ denotes a point 
$q=(e^{i\theta_1}, \dots, e^{i\theta_N})\in \Gamma_N$. 
For $h\in (0,1)$ the parameter $v$ may vary in the following interval
\begin{eqnarray}\label{interval}
-\frac{h^2}{2} \, \le \, v\, \le \, h+\frac{1}{2}.
\end{eqnarray}
However for $h\ge 1$ ({\it strong magnetic field}) the interval of variation of $v$ is actually smaller
\begin{eqnarray}\label{interval1}
-h + \frac{1}{2} \, \le \, v\, \le \, h+\frac{1}{2}.
\end{eqnarray}

%
%

Differentiating the equation (\ref{vm}) with respect to the angular variable $\theta_j$ we obtain
$$\langle \frac{\partial \MM}{\partial \theta_j}, \MM+\MM_0\rangle = \frac{1}{N}\frac{\partial V}{\partial \theta_j}$$ where the brackets $\langle\, \, , \, \rangle$ denote the 
Euclidean planar scalar product. Since clearly 
$$\frac{\partial \MM}{\partial \theta_j}= \frac{1}{N} e^{i(\theta_j+\pi/2)},$$ 
it follows that at a critical point of $V$ either  $\MM+\MM_0=0$ or, for any $j=1, \dots, N$,
the vector $e^{i\theta_j}$ is parallel to $\MM+\MM_0$. 
This implies that, for $h\not=0$, at any critical point of $V$ lying outside the submanifold $\MM+\MM_0=0$ (the ground state), 
one has 
$$\theta_j= \pm \pi/2,$$
i.e. a critical configuration is collinear; it is aligned in the direction of the imaginary axis. 

%
%
%
\begin{figure}[h]
\begin{center}
\resizebox{7cm}{6cm}{\includegraphics[71,370][520,745]{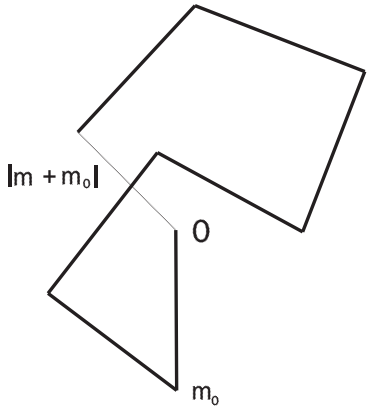}}
\end{center}
\caption{A robot arm. }\label{arm}
\end{figure}
Consider now a robot arm with $N+1$ bars (as shown on Figure \ref{arm}) where the first bar has length $\ell_1=h$ and is pointing from the origin to the point $\MM_0$ and the remaining $N$ bars are of 
length $\ell_j=1/N$, where $j=1, \dots, N$, as described in \cite[\S 1.5]{Fa3}. We assume that the initial point of the arm is fixed at the origin and the bars of the arm are connected to each other via revolving joints. 
 If $\theta_j$ is the angle between the $j$-th bar and the horizontal $x$-axis, then $\MM+\MM_0$ in (\ref{magnet}) is exactly the final position (the grip) of the arm. 

Clearly the length
$|\MM+\MM_0|$ of the dotted line in Figure \ref{arm} is closely related to the value of the robot arm distance map studied in \cite[Lemma 1.4]{Fa3}; the robot arm distance map is also described in the beginning of \S 4 below. 

Lemma 1.4 from \cite{Fa3} deals with the shapes of the arm, i.e. with the quotient of the configuration space with respect to
 the action of ${\rm {SO}}(2)$. The critical points of the robot arm distance map are the following: there is a critical submanifold $\MM+\MM_0=0$ (the ground state) which corresponds to the minimum of $V$. This submanifold has dimension $N-2$ and is diffeomorphic to the configuration space of a closed linkage with $N+1$ bars $(h, 1/N, \dots, 1/N)$. The other critical points are of Morse type; they correspond to collinear configurations of the arm, i.e. $\theta_i=\pm {\pi}/{2}$ where $i=2, \dots, N+1$. 

Translating this result into the language of thermodynamics we obtain {\it the bi-clus\-tering phenomenon} observed in \cite{AR, DHR} numerically.
It consists of the statement that under low temperature some of the rotators are likely to point in a fixed direction $\psi_0$ with the remaining rotators pointing in the opposite 
direction $\psi_0+\pi$. This follows from the fact that (a) the canonical measure for $T\to 0$ is concentrated near critical points of the energy $V$; 
(b) the critical points of the energy $V$ are the ground state $\MM+\MM_0=0$ and the Morse critical points corresponding to collinear configurations; 
(c) the ground state $\MM+\MM_0=0$ is highly degenerate. 

%


Using formula (\ref{vm}), the manifold $\mathcal M_v=\{q; V(q)\le Nv\}$ can be interpreted mechanically as the configuration space of a linkage with $N+1$ legs of fixed length 
$h, 1/N, 1, \dots, 1/N$ and one telescopic leg with its length varying between $0$ and $(2v+h^2)^{1/2}$. 
Recently we studied linkages with a telescopic leg whose length varies between two positive numbers 
 \cite{Fa4}.
In the next section, we give a modification
of the arguments of \cite{Fa4} which allows for a computation of the
Betti numbers in the important case when the telescopic leg is
allowed to contract to zero.

\section{Telescopic linkages}\label{s:3}

Motivated by the discussion of the previous section we study here the topology of configuration spaces of telescopic linkages and calculate their Betti numbers.  

Let $\ell=(\ell_1, \ell_2, \dots, \ell_n)$ be a fixed vector with positive real entries, $\ell_i>0$. Consider the variety of shapes $K_\ell$ 
of closed planar polygonal chains 
consisting of $n-1$ bars of fixed length (equal to $\ell_1, \dots, \ell_{n-1}$, correspondingly) and a telescopic leg whose length may vary between $0$ and $\ell_n>0$.
Formally $K_\ell$ is defined as follows. Consider the map $F: \C^n\to \R^n$ given by $$F(z_1, \dots, z_n)= (|z_2-z_1|, |z_3-z_2|, \dots, |z_1-z_n|), \quad z_i\in \C.$$
Then $K_\ell$ can be defined as 
$$K_\ell=F^{-1}(A)/E(2),$$ where $A\subset \R^n$ is the closed interval connecting the points 
$$\ell^-=(\ell_1, \dots, \ell_{n-1}, 0), \quad \mbox{and}\quad \ell=(\ell_1, \dots, \ell_{n-1}, \ell_n),$$ and $E(2)$ denotes the group of orientation preserving isometries of the plane $\C$,  acting diagonally on $\C^n$. Recall that 
$$M_\ell=F^{-1}(\ell)/E(2)\subset K_\ell$$
is the well-studied variety of shapes of closed polygonal chains with $n$ bars of length $\ell_1, \dots, \ell_n$; see \cite{Fa2, Fa3}.  

The manifold $K_\ell$ can also be understood as the variety of all configurations of a robot arm with $n-1$ bars of length $\ell_1, \dots, \ell_{n-1}$ such that the initial point $O$ is fixed, the first bar of length $\ell_1$ is pointing in the direction of the $x$-axis, and the end point of the arm (\lq\lq the grip\rq\rq) 
lies within a circle of radius $\ell_n$ with center at $O$, see Figure below.  
\begin{figure}[h]
\begin{center}
\resizebox{6cm}{4.1cm}{\includegraphics[37,419][555,791]{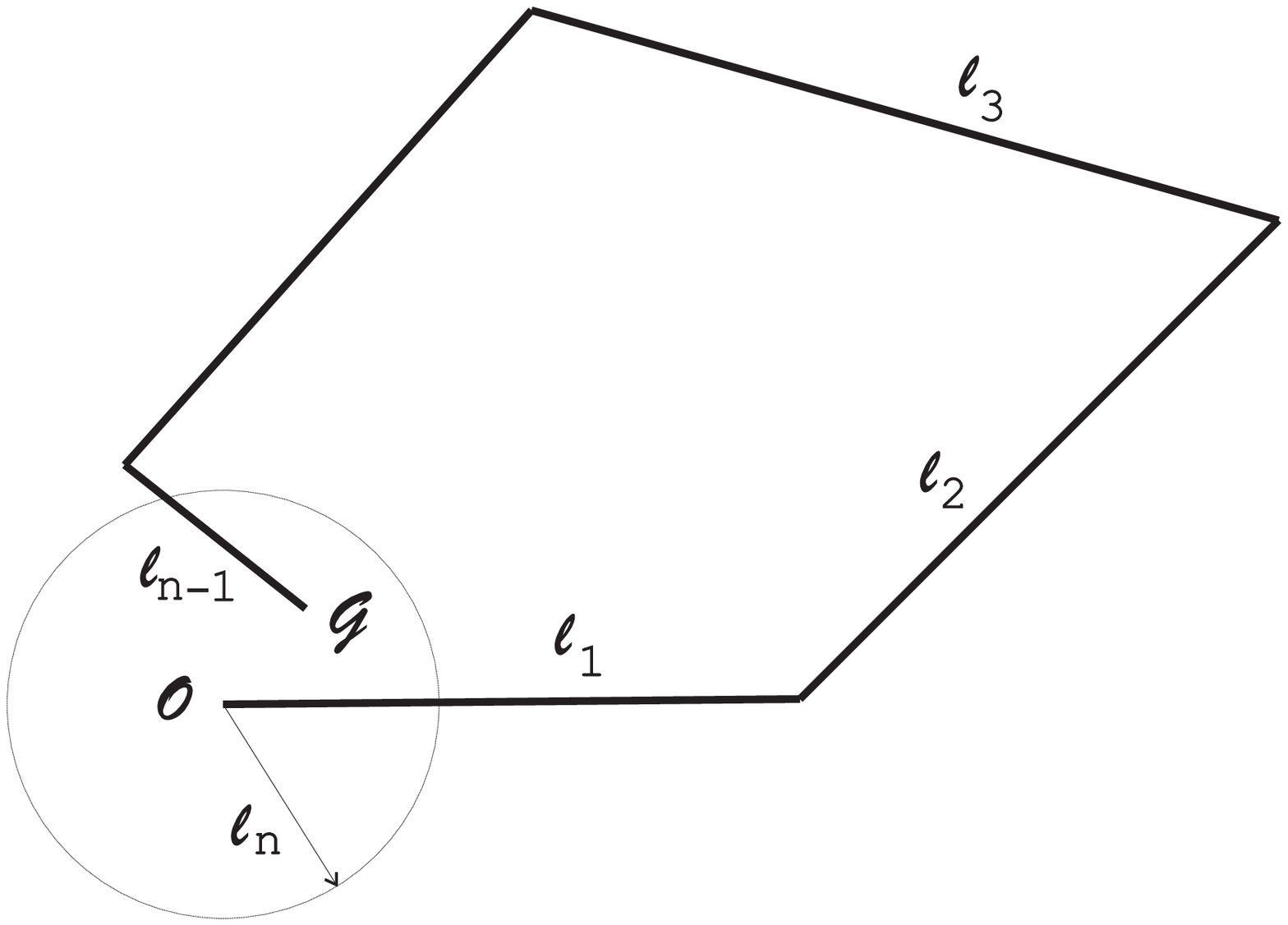}}
\end{center}
\end{figure}

Our goal in this paper is to compute the homology groups of $K_\ell$ explicitly as functions of the vector $\ell$ of metric data. In our previous paper \cite{Fa4} we considered configuration spaces of linkages having a telescopic leg not contractable to zero, i.e. such that its length may vary between two positive numbers. The results of the present paper complement those obtained in \cite{Fa4}. 

Recall that a vector $\ell= (\ell_1, \dots, \ell_n)$ is said to be {\it generic} if $\sum_{i=1}^n\epsilon_i\ell_i\not=0$ for all $\epsilon_i=\pm 1$. 

\begin{prop}\label{prop1}
If the length vector $\ell= (\ell_1, \dots, \ell_n)$ is generic then $K_\ell$ is a compact smooth $(n-2)$-dimensional manifold with boundary $\partial K_\ell=M_{\ell}$.
\end{prop}

A poof is given in the following section. 

It is clear that $K_\ell$ is diffeomorphic to $K_{\ell'}$ if the vector $\ell'=(\ell'_1, \dots, \ell'_n)$ is obtained from from $\ell$ by permuting the coordinates $\ell'_i=\ell_{\sigma(i)}$ 
where the permutation $\sigma: \{1, \dots, n\}\to \{1, \dots, n\}$ satisfies $\sigma(n)=n$. The $n$-th coordinate corresponding to the telescopic leg plays a special role. 
It follows that without loss of generality one may always assume that $\ell_1\le \ell_2\le \dots\le \ell_{n-1}$. 

Let $\ell=(\ell_1, \dots, \ell_n)\in \R^n$ be a vector with positive coordinates. A subset $J\subset \{1, \dots, n\}$ is said to be {\it short} with respect to $\ell$ if
\begin{eqnarray}\label{short}\sum_{i\in J} \ell_i \, <\,  \sum_{i\notin J} \ell_i.\end{eqnarray}
If the opposite inequality holds in (\ref{short}) then $J$ is said to be {\it long} with respect to $\ell$. A subset $J\subset \{1, \dots, n\}$ is called {\it median} if 
 $$\sum_{i\in J} \ell_i \, =\,  \sum_{i\notin J} \ell_i.$$

\begin{definition} 
Fix an index $i\in \{1, 2, \dots, n-1\}$ such that $\ell_i\ge \ell_j$ for any $j=1, \dots, n-1$.
We denote by $c_k(\ell)$ the number of $(k+1)$-element subsets $J\subset \{1, \dots, n\}$ which contain $i$, do not contain $n$ and are short or median with respect to $\ell$; here $k=0, 1, \dots, n-2$ .
Besides, for $k=1, \dots, n-2$ we denote by $d_k(\ell)$ the number of $(k+1)$-element subsets $J\subset \{1, \dots, n\}$ containing both $i$ and $n$ which are short with respect to $\ell$. 
\end{definition}

It will be convenient to extend this Definition by setting $d_k(\ell)=0$ for $k\le 0$. 

In \cite{Fa4} we introduced a symbol $\alpha_k(\ell)$ which equals the number of subsets of cardinality $k+1$ containing $n$ which are short with respect to $\ell$. 
One may express the numbers $c_k(\ell)$ and $d_k(\ell)$ introduced above through the quantities $\alpha_k$ as follows.
Assume for simplicity that $\ell_1\le \ell_2 \le \dots \le \ell_{n-1}$ (this can always be achieved by a permutation preserving the index of the telescopic leg). 
Then $d_k(\ell)$ equals
$\alpha_{k-1}(L)$ where $L=(\ell_1, \dots, \ell_{n-2}, \ell_{n-1}+\ell_n)$ is obtained by integrating the $(n-1)$-st and the $n$-th legs.
Besides, if we assume that $l$ is generic then $c_k(\ell)=\alpha_k(L')$ where $L'=(\ell_1, \dots, \ell_{n-2}, \ell_{n-1}-\ell_n)$.

\begin{thm}\label{thm1} For a telescopic linkage described above with a generic vector $\ell$, the homology group $H_k(K_\ell)$ of the configuration space $K_\ell$ is free abelian of rank 
\begin{eqnarray}\label{theresult}
c_k(\ell)+d_{n-3-k}(\ell), \quad k=0, 1, \dots, n-2.
\end{eqnarray}
\end{thm}

A proof is given in the following section.

\begin{exa}{\rm Consider the zero-dimensional Betti number $b_0(K_\ell)= c_0(\ell)+d_{n-3}(\ell)$. As above, we assume that 
$\ell_1\le \dots\le \ell_{n-1}$. 
Clearly, $c_0(\ell)$ equals $1$ iff the singleton ${n-1}$ is short or median with respect to $\ell$; otherwise $c_0(\ell)=0$. Moreover, for $n>3$ the number $d_{n-3}(\ell)$ equals $1$ iff the set $\{n-2, n-3\}$ is long with respect to $\ell$; 
otherwise $d_{n-3}(\ell)=0$. We obtain that the manifold $K_\ell$ is disconnected if and only if 
\begin{eqnarray}\ell_{n-3}+\ell_{n-2}> \frac{1}{2}\sum_{i=1}^n \ell_i.\end{eqnarray}
}
\end{exa}

\begin{cor}\label{cor1} For $n>3$ the following conditions are equivalent:
\begin{enumerate}
  \item $K_\ell$ is disconnected;
  \item $K_\ell$ consists of two connected components;
   \item The set $\{n-3, n-2\}$ is long with respect to $\ell$;
  \item $K_\ell$ is diffeomorphic to the disjoint union 
$(T^{n-4}\times D^2)\sqcup (T^{n-4}\times D^2).$
 \end{enumerate}
\end{cor}

\begin{cor}\label{cor2} Suppose that $\ell_n>0$ is small in the following sense: for any choice of $\epsilon_j=\pm 1$ where $j=1, \dots, n-1$, such that 
$\sum_{j=1}^{n-1} \epsilon_j \ell_j >0,$
one has 
$\sum_{j=1}^{n-1}\epsilon_j \ell_j >\ell_n.$ Then $K_\ell$ is homotopy equivalent to  $M_{\ell'}$ where $\ell'=(\ell_1, \dots, \ell_{n-1})$. 
\end{cor}

We will skip the proofs of Corollaries \ref{cor1} and \ref{cor2} since they are analogous to the corresponding statements in \cite{Ka1, Fa4}. 

It is easy to see that under the assumptions of Corollary \ref{cor2} the Betti numbers $b_k(K_\ell)$ coincide with those given by \cite{Fa2}.

\section{Proofs of Proposition \ref{prop1} and Theorem \ref{thm1}} \label{s:4}

Consider the torus $T^{n-1}=S^1 \times \dots \times S^1$ and the quotient $W=T^{n-1}/{\rm {SO(2)}}$ with respect to the diagonal action of the rotation group ${\rm {SO(2)}}$. Clearly, $W$ can be identified with the torus of dimension $n-2$. Consider the function $f: W\to \R$ given by
$$f(u_1, \dots, u_{n-1}) = - \left|\sum_{i=1}^{n-1}\ell_iu_i\right|^2, \quad u_i\in S^1.$$
Then $K_\ell$ can be identified with the preimage
$K_\ell\simeq f^{-1}[a',0],$ where $a'=-(\ell_n)^2.$ 

We know that the critical points of $f$ consist of the preimage $f^{-1}(0)=M_{\ell'}$, where $\ell'=(\ell_1, \dots, \ell_{n-1})$, and finitely many Morse critical points corresponding to collinear configurations, see \cite{Fa2}. It follows that if the vector $\ell$ is generic then $a'$ is a regular point of $f$ and  therefore the preimage $f^{-1}[0,a']$ is a smooth manifold with boundary.
This proves Proposition \ref{prop1}.

In the following arguments we do not assume that $\ell$ is generic. In other words, $a'$ is not necessarily a regular value of $f$. 
Let $a<a'$ be a regular value of $f$ such that the interval $[a,a')$ contains no critical values. 

Denote $W^{a}=f^{-1}(-\infty, a]$. Consider the long exact sequence
$$\to H_{k+1}(W) \to H_{k+1}(W, W^{a})\to H_k(W^{a}) \stackrel{j_k}\to H_k(W) \to \dots$$
where $j_k$ is induced by the inclusion $j: W^{a}\to W$. 
Using excision and Poincar\'e duality we obtain 
$$H_{k+1}(W, W^{a}) \simeq H_{k+1}(f^{-1}[a,0], \partial f^{-1}[a,0])\simeq H^{n-3-k}(f^{-1}[a,0]) \simeq H^{n-k-3}(K_\ell).$$ 
Here we used the observation that $K_\ell\subset f^{-1}[a,0]$ is a deformation retract. 

Thus we obtain the short exact sequence
\begin{eqnarray}\label{exseq}
0\to \coker ( j_{k+1}) \to H^{n-3-k}(K_\ell) \to \ker (j_k) \to 0.
\end{eqnarray}
Therefore to compute the cohomology of $K_\ell$ it is enough to find the kernel and cokernel of the homomorphism $j_k$; note that $H_\ast(W^a)$ and $\ker j_k$ are torsion free
(see below) and therefore the exact sequence (\ref{exseq}) splits. 

Next we describe the homology of the manifold $W^a$ 
following \cite{Fa2,Fa3,Fa4}.
For any subset $J\subset \{1, \dots, n-1\}$ consider the subset $W_J\subset W\simeq T^{n-2}$ consisting of all configurations $(u_1, \dots, u_{n-1})$ such that $u_i=u_j$ for all $i, j\in J$.
In other words, we \lq\lq freeze\rq\rq\,  all links labeled by indices in $J$ to be parallel to each other. It is clear that $W_J$ is diffeomorphic to a torus of dimension $n-1-|J|$.

The torus $W_J$ is contained in $W^{a}$, i.e. $W_J\subset W^{a}$, if and only if $J$ (viewed as a subset of $\{1, \dots, n\}$) is long with respect to $\ell$.
Indeed, let $p_J=(u_1, \dots, u_{n-1})$ be the configuration where $u_i=1$ for all $i\in J$ and
$u_i=-1$ for all $i\notin J$. Then the maximum of the restriction $f|W_J$ is either $0$ or $f(p_J)$, see \cite[Lemma 8, statement (4)]{Fa2}. The inequality $f(p_J)\le a$ 
is equivalent to $$\sum_{i\in J} \ell_i >\frac{1}{2} \sum_{i=1}^n\ell_i$$
which means that $J$ is long with respect to $\ell=(\ell_1, \dots, \ell_n)$.

One may fix naturally orientations of $W$ and all submanifolds $W_J$, see \cite{Fa4}. However in this paper we will not need to deal with specific orientations and will assume 
that the manifolds $W_J$ and $W$ are somehow oriented. The class $[W_J]\in H_k(W)$ is then well-defined where $|J|=n-1-k$. 

If $k+k'=n-2$ and $J, J'\subset \{1, \dots, n-1\}$ are 
subsets with $|J|=n-1-k$, $|J'|=n-1-k'$ then the intersection number $[W_J]\cdot[W_{J'}]$ is either zero (iff $|J\cap J'|>1$) or $\pm1$ (iff $|J\cap J'|=1$); see \cite[formula (33)]{Fa2}.
Note that $k+k'=n-2$ implies that $|J|+|J'|=n$, i.e. the subsets $J$ and $J'$ must have a nontrivial intersection.

Without loss of generality we may assume that $\ell_1\le \ell_2\le \dots \le \ell_{n-1}$. 

One observes that the homology classes realized by the tori $W_I\subset W$, 
where $I$ runs over all subsets $I\subset \{1, \dots, n-1\}$ of cardinality $n-1-k$ containing $n-1$, form a free basis of the homology group $H_k(W)$. We will write
\begin{eqnarray}
H_k(W) =A_k\oplus B_k
\end{eqnarray}
where $A_k$ is generated by the homology classes $[W_J]\in H_k(W)$ with $n-1\in J$, $|J|=n-1-k$ and such that $J$ is long with respect to $\ell$. The subgroup $B_k$ is generated by the classes $[W_J]$ with $J$ satisfying $n-1\in J$, $|J|=n-1-k$ and $J$ is short or median with respect to $\ell$.

By \cite[Corollary 9]{Fa2} the homology classes $[W_J]$ of the submanifolds $W_J$, 
where $J$ runs over all subsets $J\subset \{1, \dots, n-1\}$ of cardinality $n-1-k$ which are long with respect to $\ell$, 
form a basis of the free abelian group $H_k(W^{a})$. This conclusion is based on the technique of Morse theory in the presence of an involution as developed in \cite{Fa2}. 

We will write
\begin{eqnarray}
H_k(W^{a}) =A_k \oplus C_k
\end{eqnarray}
where $A_k$ is defined above and $C_k$ is generated by the classes $[W_J]\in H_k(W^a)$ such that $J\subset \{1, \dots, n-2\}$ is long with respect to $\ell$, and 
$|J|=n-1-k$. 

Note that $C_k=0$ for $k=0$. 

Next we may analyze the homomorphism $j_k: H_k(W^{a}) \to H_k(W)$ induced by the inclusion $W^{a} \subset W$. It is obvious that $j_k[W_J]=[W_J]$ assuming that 
$[W_J]\in A_k$. 

We claim that if $[W_J]\in C_k$ then $j_k[W_J]$ is a linear combination of the classes $[W_K]\in A_k$. Indeed, we may write
\begin{eqnarray}\quad
j_k[W_J]= \sum_{[W_I]\in A_k}a_I\cdot [W_I] + \sum_{[W_K]\in B_k}b_K\cdot [W_K], \quad a_I, b_K\in \Z.
\end{eqnarray}
Then, for $[W_K]\in B_k$, the coefficient $b_K$ equals 
\begin{eqnarray}
b_K= \pm [W_J]\cdot [W_{K'}],
\end{eqnarray}
where $K'$ is obtained from the complement of $K$ in $\{1, \dots, n-1\}$ by adding $n-1$. If $b_K\not=0$ then $J\cap K'=\{j\}$ is a single element and hence $K$ is obtained from 
$J$ by removing $j\in J$ and adding $n-1$. However since $J$ is long with respect to $\ell$ and $\ell_j\le \ell_{n-1}$ it follows that $K$ is long with respect to $\ell$ as well, which contradicts the assumption $[W_K]\in B_k$.  

From the exact sequence (\ref{exseq}) we find that $H^{n-3-k}(K_\ell)$ is free abelian of rank 
\begin{eqnarray}\label{ranks}
\rk\, H^{n-3-k}(K_\ell) =\rk\, H_{n-3-k} (K_\ell)= \rk\, C_k +\rk B_{k+1}.
\end{eqnarray}
Now we may calculate the ranks of $C_k$ and $B_{k+1}$. Examining the definition we see that the rank of $C_k$ equals the number of 
$(k+1)$-element
subsets $J\subset \{1, \dots, n\}$ containing both 
$n-1$ and $n$ which are short with respect to $\ell$. In other words, $\rk\,  C_k=d_k(\ell)$. 
The rank of $B_{k+1}$ equals the number of $(n-2-k)$-element subsets $J\subset \{1, \dots, n-1\}$ which contain $n-1$ and are short or median with respect to $\ell$, i.e. 
$\rk\,  B_{k+1}= c_{n-3-k}(\ell)$. Substituting into (\ref{ranks}) we obtain (\ref{theresult}). \qed

\section{The total Betti number phase transition in the anti-ferromagnetic  mean-field $XY$ model}\label{s:5}

Consider again the anti-ferromagnetic mean-field $XY$ model as discussed in \S 2. 
According to the discussion at the end of \S 2 and by Theorem \ref{thm1}, the Betti numbers $b_k(\mathcal M_v)$ are equal to 
$c_k(\ell)+d_{n-3-k}(\ell)$ with the length vector $\ell$ being the following
$$\ell=(\epsilon_N, \epsilon_N, \dots, \epsilon_N, h, (2v+h^2)^{1/2}),$$
where $\epsilon_N = 1/N$, $n= N+2$ and $h$ denotes the magnetic field. We assume that $h$ is positive and constant, i.e. independent of $N$. 

Our goal is to examine the total Betti numbers 
\begin{eqnarray}\label{bmv}
b({\mathcal M}_v)= \sum_{k=0}^{n-2}b_k({\mathcal M}_v) = c(\ell) + d(\ell)\end{eqnarray}
and their rate of exponential growth
\begin{eqnarray}\label{tauv}
\tau(v)= \lim _{n\to \infty} \frac{\ln b({\mathcal M}_v)}{n}.
\end{eqnarray}
We want to investigate the possibility that singularities of this quantity (which might be more sensitive than (\ref{sigmav})) will allow detection of the phase transition. 
Here $c(\ell)$ denotes $\sum_{k}c_k(\ell)$; if $N$ is large\footnote{i.e. when $h>1/N$.} $c(\ell)$ is the number of all subsets of $\{1, \dots, n\}$ containing $n-1$ and not containing $n$ which are short or median with respect to $\ell$. 
Similarly, $d(\ell)$ denotes $\sum_{k}d_k(\ell)$; for $N$ large $d(\ell)$ it is the number of all subsets of $\{1, \dots, n\}$ containing $n-1$ and $n$ which are short with respect to $\ell$. 

\begin{definition} We will say that the system undergoes {\it a total Betti number phase transition} at $v=v_0$ if the function $\tau(v)$ given by the formula (\ref{tauv}) is not analytic at 
$v=v_0$. 
\end{definition}

We motivate this definition by the topological hypothesis as presented in section \cite[section V.A]{Kastner04}. Namely, it is known that in many cases non-analyticity of the function 
$\sigma(v)$ mentioned in the introduction detects phase transitions of the system. The quantity $\tau(v)$ is defined similarly to $\sigma(v)$, but using the total Betti number rather than the Euler characteristic. 

To illustrate the behavior of $c(\ell)$ and $d(\ell)$, consider first the case where the magnetic field is strong, namely $h>1$. Then every subset containing both the indices $n-1$ and $n$ is long and we obtain $d(\ell)=0$. On the other hand, the value of $c(\ell)$ in this case still depends on $v$. 


Recall that the parameter $v$ varies in the interval $[a_h, b_h]$ where 
$$a_h= \left\{ \begin{array}{lll}
-\frac{1}{2}h^2,& \mbox{for}& h\in (0,1],\\ \\
-h+\frac{1}{2}, &\mbox{for}& h\in [1, \infty),
\end{array}\right.
$$
and $b_h=h+1/2$, see \S 2. 
Let  $p_{v}$ denote the following quantity
$$p_v=\frac{1}{2}((2v+h^2)^{1/2}-h+1).$$
It is easy to check that for $v\in (a_h, b_h)$ one has 
$0< p_v< 1$. Besides, zero belongs to the interval $(a_h, b_h)$ and for $v=0$ one has $p_v=1/2$. 
\begin{thm} For $h > 0$ and $v\in (a_h, b_h)$ the rate of exponential growth of the total Betti number (\ref{tauv}) equals
\begin{eqnarray}\label{rate}
\tau(v) = \left\{
\begin{array}{lll}
-p_v\ln p_v -(1-p_v)\ln(1-p_v), & \mbox{for}& v \le 0,\\ \\
\ln 2, &\mbox{for} & v\ge 0.
\end{array}\right.
\end{eqnarray}
In particular, the function $\tau(v)$ and its first derivative are continuous but the second derivative of $\tau$ is discontinuous at $v=0$. In other words, the system undergoes a total Betti number phase transition at $v=0$. 
\end{thm}

\begin{proof} It will be convenient to use the notations 
$$S^n_k=\sum_{0\le i\le k}\binom n i \quad \mbox{and}\quad R^n_k=\sum_{0\le i<k}\binom n i.$$ Using formula (\ref{bmv}) we may write
$$b({\mathcal M}_v)= S^{n-2}_{p_v(n-2)} + R^{n-2}_{(1-p_v-h)(n-2)}.$$
If $v\le 0$ then $1-p_v-h \le p_v\le 1/2$ and therefore
\begin{eqnarray}\label{ineq}
\quad \binom {n-2} {[p_v(n-2)]}\, <\,  b({\mathcal M}_v)\, <\,  2\cdot S^{n-2}_{p_v(n-2)}\le n \cdot \binom {n-2} {[p_v(n-2)]}.
\end{eqnarray}

We will use the following well-known asymptotic formula for the binomial coefficients
$$\binom n m \sim (2\pi)^{-1} \left(\frac {n}{m}\right)^{m}\cdot \left( \frac{n}{n-m}\right)^{n-m} \cdot\left[\frac{m(n-m)}{n}\right]^{-1/2}$$
which is valid if $n, m\to \infty$ and $n-m\to \infty$, see \cite[page 4]{B}. The meaning of the symbol $f(n)\sim g(n)$ is $\lim f(n)/g(n)=1$. 

After some elementary calculations the asymptotic formula above gives
$$\lim_{n\to \infty} \frac{1}{n} \ln \binom {n-2} {[p_v(n-2)]} \, = -p_v\ln p_v -(1-p_v)\ln (1-p_v).$$
Now, the inequalities (\ref{ineq}) imply the first part of (\ref{rate}). 

If $v\ge 0$ then $p_v\ge 1/2$ and $S^{n-2}_{p_v(n-2)} \ge 1/2 \cdot 2^{n-2}= 2^{n-3}.$ Therefore, in this case 
$$2^{n-3} \le b({\mathcal M}_v)\le 2^{n-1}$$
implying $\tau(v)= \ln 2$. This gives the second part of formula (\ref{rate}). 
\end{proof}

\section{Conclusions}\label{s:6}

In this paper we exploited an interpretation of the sub-energy manifolds of the anti-ferromagnetic mean-field $XY$ model as configuration spaces of linkages with one telescopic leg. 
Using Morse theory techniques, enriched with implications which stem from the presence of an involution, we gave a complete computation of the Betti numbers of the sub-energy manifolds.  

As an indicator of phase transitions we studied the exponential growth rate of the total Betti number as opposed to the exponential growth rate of the Euler characteristic, 
as studied by the previous authors. 

We showed by an explicit computation that in the case of non-vanishing magnetic field there is a unique total Betti number phase transition in the anti-ferro\-magnetic mean-field $XY$ model. 

%

We hope that using the total Betti number instead of the Euler characteristic might provide a more sensitive tool for the study of different versions of the topological hypothesis. 
We suggest that the behaviour of the exponential growth rate of the total Betti number of the sub-energy manifolds and its relationship to the physical properties of the system be examined in various models.

\vskip 1cm

The authors are thankful to the anonymous referee for a number of valuable comments and suggestions. 


\newpage

\begin{thebibliography}{99}


\bibitem{Angelani2005} L. Angelani, L. Casetti, M. Pettini, G. Ruocco and F. Zamponi,
`Topology and phase transitions: From an
exactly solvable model to a relation between topology and
thermodynamics', {\it Phys. Rev. E}  {\bf 71} (2005), 036152 (112).

\bibitem{AR} M. Antoni and S. Ruffo, `Clustering and relaxation in Hamiltonian long-range dynamics', {\it Phys. Rev. E} {\bf 52} (1995), 2361--2374. 



\bibitem{B} B. Bollob\'{a}s, {\it Random Graphs}, Second edition, Cambridge Stud. Adv. Math., 73 (Cambridge University Press, Cambridge, 2001).




\bibitem{CPC2003} L.Casetti, M. Pettini and E.G.D. Cohen, `Phase transitions and topology changes in configuration space', {\it J. Stat. Phys.} {\bf 111} (2003), 1091--1123.



\bibitem{DHR} T. Dauxois, P. Holdworth and S. Ruffo, `Violation of ensemble equivalence in the antiferromagnetic mean-field XY model', {\it Eur. Phys. J. B} {\bf 16} (2000), 659--667.
%



\bibitem{Fa2} M. Farber and D. Schuetz, `Homology of Planar Polygon Spaces', {\it Geom. Dedicata} {\bf 125} (2007), 75--92.
\bibitem{Fa3} M. Farber, \textit{Invitation to Topological Robotics}, Zurich Lectures in Advanced Mathematics (European Mathematical Society, 2008).
\bibitem{Fa4} M. Farber and V. Fromm, `Homology of planar telescopic linkages', {\it Algebraic and Geometric Topology} {\bf 10} (2010), 101--125



\bibitem{Farber2004}  M. Farber, {\it Topology of closed one-forms}, Mathematical Surveys and Monographs, 108 (American Mathematical Society, 2004).
%
%
%
%
%
%
%
%
%


 \bibitem{FP2004} R. Franzosi and M. Pettini, `Theorem on the origin of the phase transitions', {\it Phys. Rev. Lett.} {\bf 92} (2004), 060601.
%
%
%

\bibitem{Ka1} M. Kapovich and J. L. Millson, `On the Moduli Space of Polygons in the Euclidean Plane', {\it J. Diff. Geom.} {\bf 42} (1995), 133--164.
%





%
%
%


\bibitem{Kastner04} M. Kastner, `Unattainability of a purely topological criterion for the existence of a phase transition in nonconfining potentials', {\it Phys. Rev. Lett.} {\bf 93} (2004), 150601.

\bibitem{Kastner08} M. Kastner, `Phase transitions and configuration space topology', {\it Rev. Mod. Phys.} {\bf 80} (2008), 167--187.

\bibitem{LCJ} D.H. Lee, R.G. Caflisch and J.D. Joannopoulos, `Antiferromagnetic classical $XY$ model: A mean-field analysis', {\it Phys. Rev. B} {\bf 29} (1984), 2680--2684.

\bibitem{Pettini} M. Pettini, `Geometry and Topology in Hamiltonian Dynamics and Statistical Mechanics', Interdisciplinary Applied Mathematics (Springer, New York, 2007).

\bibitem{Teixiera} A. C. R. Teixeira and D. A. Stariolo, `Topological hypothesis on phase transitions: The simplest case', {\it Phys. Rev. E} {\bf 70} (2004), 016113.


\end{thebibliography}
\end{document}